\documentclass[11pt]{article} 

\usepackage[utf8]{inputenc} 

\usepackage{geometry} 
\usepackage{graphicx} 
\usepackage{enumerate}

\usepackage{amsmath}
\usepackage{amssymb}
\usepackage{amsthm}

\newtheorem{lemma}{Lemma}
\newtheorem{theorem}{Theorem}
\theoremstyle{remark}
\newtheorem{remark}{Remark}
\usepackage{amsmath} \usepackage{amssymb} 
\usepackage{booktabs} 
\usepackage{array} 
\usepackage{paralist} 
\usepackage{verbatim} 
\usepackage{subfig} 
\usepackage{url}

\usepackage{fancyhdr} 
\usepackage{sectsty}
\allsectionsfont{\sffamily\mdseries\upshape} 

\usepackage[nottoc,notlof,notlot]{tocbibind} 
\usepackage[titles,subfigure]{tocloft} 

\title{Schatten-class Truncated Toeplitz Operators}
\author{Patrick Lopatto\thanks{The first author was supported by the National Science Foundation under grant DMS-1055897.} \and Richard Rochberg\thanks{The second author was supported by the National Science Foundation under grant DMS-1001488.}}

\begin{document}
\maketitle

\section{Introduction}

We let $H^2$ denote the Hardy space $H^2(\mathbb D)$ of the open unit disk $\mathbb D$ in the complex plane. Recall that $H^2$ is the Hilbert space of holomorphic functions which are power series centered at $0$ with square-summable coefficients. As usual, we identify $H^2$ with its space of boundary values, the subspace of $L^2(\partial \mathbb D)$ spanned by $\{1,z, z^2,\dots\}$, where the functions in this set are considered as functions on $ \partial \mathbb D$. We let $P$ denote the Cauchy projection, the orthogonal projection from $L^2(\partial \mathbb D)$ to $H^2$. This projection may be extended in a natural way to a map from $L^1(\partial \mathbb D)$ to $H(\mathbb D)$, the set of holomorphic functions on $\mathbb D$ (see \cite{Sarason2007}). 

A (possibly unbounded) classical Toeplitz operator is defined by starting with a function $\phi\in L^2 (\partial \mathbb D)$, called a symbol function, and compressing the multiplication operator $M_\phi$ to $H^2$. That is, we define $$T_\phi= PM_\phi.$$
It is known that $T_\phi$ is bounded if and only if $\phi\in L^\infty$, and in general it is interesting to relate properties of $T_\phi$ to those of the symbol function $\phi$. Extensive work has been done on classical Toeplitz operators and much is now known about them \cite{Nikolski2002}.

Recently, Sarason has proposed studying compressions of classical Toeplitz operators to coinvariant subspaces of the shift operator on $H^2$ \cite{Sarason2007}. He calls these truncated Toeplitz operators, which we will abbreviate as TTO. Recall that by Beurling's theorem, any coinvariant subspace of the shift operator is of the form 
$$K_u=H^2 \circleddash uH^2$$
for some inner function $u$. Such spaces are also called model spaces, and we denote the orthogonal projection to the model space $K_u$ by $P_u$. Every $K_u$ comes equipped with an anti-unitary conjugation operator $C$, whose properties are discussed in detail in \cite{Sarason2007}. Here we just recall that, given $f\in K_u$, $(Cf)(e^{i\theta})=u(e^{i\theta})\overline{e^{i\theta} f(e^{i\theta})}$. Note that this is an equality of boundary values, not interior values. Also, we will use the same notation and formula for general $f\in L^2$. 

Given a model space $K_u$ and a symbol function $\phi\in L^2$, we define the truncated Toeplitz operator $A_\phi$ on $K_u$ as 
$$A_\phi=P_uM_\phi.$$
As before, the general project is to deduce properties of $A_\phi$ from properties of $\phi$ and vice versa. A TTO does not have a unique symbol. However, every TTO has a unique symbol in $K_u + \overline K_u$  \cite{Sarason2007}, and we will often focus attention on that choice of symbol.

Hankel operators on $H^2$ are closely related to Toeplitz operators. The Hankel operator $H_\phi$ with symbol $\phi\in L^2$ is defined as 
$$H_\phi=(I-P)M_\phi.$$
As before, this may only be densely defined. Unlike a Toeplitz operator, which is an operator from $H^2$ to $H^2$, a Hankel operator is a map from $H^2$ to $(H^2)^\perp$. If we define $P_{\overline u}$ to be the projection onto $\overline{K_u}$, we can define the truncated Hankel operator (THO) $B_\phi$ as 
$$B_\phi=P_{\overline u}H_\phi.$$
This is a map from $K_u$ to $\overline {(K_u)_0}$, the subset of $\overline{K_u}$ that vanishes at $0$. Note that if we assume $\phi \in uH^2+\overline{uH^2}$, then $B_\phi = H_\phi P_u$, since $H_\phi$ already maps into $\overline{K_u}$. 

We note, but will not use, the fact that the discussion of Hankel operators in this context could be recast in the language of Hankel bilinear forms, as in \cite{Janson1987}.

The goal of this paper is to give criteria for TTOs and THOs to be in the Schatten ideals $S_p$. Given a bounded operator $T$ between two (possibly distinct) Hilbert spaces, the singular values $\lambda_i$ are defined as the eigenvalues of the positive operator $\sqrt{T^*T}$. The Schatten $p$-norm of $T$ is then the $l^p$ norm of the sequence of singular values, and $T$ is said to be in $S_p$ if this norm is finite. The special cases $S_1$, the trace class operators, and $S_2$, the Hilbert-Schmidt operators, are well known. In the case of classical Toeplitz operators, it is known that no nonzero $T_\phi$ is in any $S_p$ ideal. The classical $S_p$ Hankel operators have been completely characterized, and this characterization is surveyed in $\cite{Zhu}$. 

We now give an informal overview of this paper's contents. Given a TTO $A_f$, we may write the symbol function $f$ as the sum of a holomorphic function and an anti-holomorphic function: $f=\phi+\overline{\psi}$. In this case, $A_f=A_\phi + A_{\overline{\psi}}=A_\phi+A^*_{{\psi}}$. This decomposition suggests that perhaps it suffices to study operators $A_\gamma$ with $\gamma$ holomorphic.

However, problems arise. First, the symbol $f$ is not unique. This difficulty is overcome by restricting attention to a canonical choice of symbol. As noted previously, Sarason showed that every TTO corresponds to a unique symbol in $K_u + \overline {K_u}$. Even in that case, the splitting into $\phi+\overline{\psi}$ is not unique. That is a minor technical issue which is resolved by requiring that $\psi$ be orthogonal to the projection onto $K_u$ of the constant function.

There is then a more fundamental problem. We are interested in knowing how the size of $A_f$ is related to the smoothness of $f$. However, passing from $A_f$ to the summands $A_\phi$ and $A_\psi^*$ does not always respect this relationship. In \cite{Baranov2010}, the authors show there is a TTO $A_f$ which is a rank one operator, but for no choice of splitting are the summands $A_\phi$ and $A_\psi^*$ bounded. Thus in recombining $A_\phi$ and $A_\psi^*$ to recapture $A_f$, the cancellation between the two terms (and subsequent loss of information) can be the primary effect. 

Our first set of results concern TTOs $A_\phi$ with holomorphic symbol $\phi$. In that case, we use Lemma 1 below to recast questions about the Schatten class membership of $A_\phi$ as questions about classical Hankel operators on the Hardy space. Using the classical Hardy space theory we then obtain in Theorem 1 conditions for $A_\phi$ to be in a Schatten class in terms of the membership of a transform of $\phi$ in a Besov space.

Even if a TTO has a holomorphic symbol, the symbol is not uniquely determined. However, if the model space is generated by a Blaschke product with zero set $Z$, then the values of the analytic symbol on $Z$ are uniquely determined. In Theorem 3 we show that if $Z$ is an interpolating sequence, then the summability properties of those values determine the Schatten class properties of the corresponding TTO. 

The next set of results involves more restricted situations in which we can make progress by studying operators without assuming the symbol is analytic. We study model spaces generated by Blaschke products associated with thin sequences, model spaces generated by certain types of singular inner functions, and operators associated with a class of very smooth symbols.

In the final section we use Hilbert space techniques to give a characterization of TTOs in the Hilbert-Schmidt class. We work with an equivalent formulation involving truncated Hankel operators. The condition obtained is in the spirit of the classical characterization of Hilbert-Schmidt Hankel operators by membership of their symbol in the Dirichlet space. 

We note that similar questions have been addressed from a different perspective by R.V. Bessonov in \cite{Bessonov2014}.

The first author would like to thank John E. McCarthy for his many helpful suggestions. 

\section{Using Classical Results}

In this section, we exploit the connection between TTOs and classical Hankel operators to give a complete description of the $S_p$ TTOs with holomorphic symbol in terms of a Besov space condition. 

\begin{lemma}
If $\phi \in K_u$ and $A_\phi$ is bounded, then $A_\phi = U(B_{\overline{C\phi}} + R)$, where $U$ is a unitary operator independent of $\phi$ and $R$ is a bounded rank one operator. 
\end{lemma}
\begin{proof}
We define $Rf=\langle \overline {C\phi} f, 1 \rangle = \langle f, C\phi \rangle$, which is bounded by the Cauchy-Schwarz inequality. We define $U=M_{u\overline z}$. Note that $Uf = C\overline f$, so $U$ is a unitary operator from $\overline {(K_u)}$ to $K_u$.

From lemma 2.1 of \cite{Sarason2007}, we know that $CA_{\overline{\phi}} C=A_\phi$. Hence 
$$A_\phi f = CA_{\overline{\phi}} C f=CP_{K_u}(u\overline{z\phi f})=CP_{H^2}(u\overline{z\phi f})=u\overline{z P_{H^2}(u\overline{z\phi f})}$$
$$=U(P_{\overline{H^2}}(\overline{u}z\phi f)) = U(P_{\overline{H^2}}(\overline{C\phi} f))=U(B_{\overline{C\phi}} + R)(f) .$$
\end{proof}
\begin{lemma}
Let $\overline{\psi}$ be an anti-holomorphic function with $\psi\in K_u$ and $\langle \psi, P_{K_u} 1 \rangle=0$, and such that $A_{\overline \psi}$ is bounded. Then $A_{\overline \psi} = U(B_{\overline{(z\overline{u\psi})}}+V)$, where $U$ is the unitary operator in Lemma $1$ and $V$ is a rank one operator.
\end{lemma}
\begin{proof}
Note that $\langle \psi, P_{K_u} 1 \rangle=0$ implies $\langle \psi, 1\rangle =0$, so $\psi(0)=0$ and we may write $\psi = z\nu$. Then 
$$A_{\overline \psi} f = C A_\psi Cf = CP_{K_u}(\psi u \overline{zf}) = u\overline{z}\overline{P_{K_u}(\nu u \overline f)}=UP_{\overline K_u}(\overline{\nu u} f )=U(B_{\overline{\nu u}}+V)(f).$$
\end{proof}
\begin{remark}
Lemma 1 shows that to study properties related to the size of $A_\phi$ (finite rank, compact, $S_p$, etc), it suffices to study the properties of $B_{\overline{C \phi}}$. In fact, noting that $\phi$ is analytic, we may study $H_{\overline{C \phi}}$. This is because, for $f\in K_u$, $B_{\overline{f}}$ extends by zero to the rest of $H^2$. If $g\in uH^2$, then $g=ug_1$, and $H_{\overline f} g = (I-P) \overline{f}ug_1$. Though the product  $\overline{f}ug_1$ may not be in $H^2$, being the product of two $L^2$ functions $\overline f$ and $ug_1$, it is in $L^1$. In fact, $\overline{f}u\in H^2$, so $\overline{f}ug_1\in H^1$ and $(I-P)\overline{f}ug_1=0$ as desired, since the Cauchy projection is the identity on $H^1$.  

In other words, if we decompose $H^2$ as $$K_u \oplus uH^2,$$ and if $\phi$ is holomorphic, then $H_{\overline{C\phi}}$ decomposes as $$H_{\overline{C\phi}}=B_{\overline{C\phi}}\oplus 0.$$
\end{remark}
\begin{theorem}
Suppose $\phi$ is analytic and $p\in (0, \infty)$. Then $A_\phi \in S_p$ if and only if $C\phi$ is in the Besov space $B_p$, and these have comparable norms in their respective spaces.
\end{theorem}
\begin{proof}
The case $p\in[1,\infty)$ follows from Remark 1 and the characterization of $S_p$ Hankel operators in \cite{Zhu} (Theorem 9.4.13). The case $p\in (0,1)$ follows similarly from the Main Theorem in \cite{Semmes1984}.
\end{proof}

\begin{theorem}
If $\phi$ is analytic, then $A_\phi$ is compact if and only if $C\phi\in VMOA$. Alternatively, $A_\phi$ is compact if and only if there exists a continuous function $g$ on $\partial D$ such that $H_{\overline {C\phi}}=H_g$.
\end{theorem}
\begin{proof}
This follows from Remark 1 and theorems 9.3.2 and 9.3.4 in \cite{Zhu}. 
\end{proof}

\begin{remark}
Note that the condition that $C\phi$ be in $VMOA$ is a subtle one, because $C\phi$ involves a factor of $u$ which is highly oscillatory and $VMOA$ is defined by the smallness of the oscillation.
\end{remark}

\begin{remark}
Since $A^*_\phi=A_{\overline{\phi}}$, the above theorems all have analogues for TTOs with conjugate-analytic symbols. In fact, these theorems suffice to characterize the $S_p$ operators in all Sedlock algebras ${\mathcal B}^\alpha$ with $\alpha \notin \mathbb T$. The Sedlock algebras, defined in \cite{Sedlock2010}, are precisely the maximal algebras of bounded TTOs. We use the notation and results in \cite{Sedlock2010}. The operators in ${\mathcal B}^\alpha$ are exactly the adjoints of operators in $\mathcal B^{\overline{\alpha}^{-1}}$, so it suffices to consider $\alpha\in \mathbb D$. Let $u_\alpha=(u-\alpha)/(1-\overline{\alpha}u)$ for $\alpha \in \mathbb D$. Then the Crofoot transform $T_\alpha$ implements a unitary equivalence between operators in $\mathcal B^\alpha$ on $K_u$ and truncated Toeplitz operators with holomorphic symbol on $K_{u_\alpha}$:
$$T_\alpha A_\phi^{u_\alpha}T^{-1}_\alpha = A^u_{\phi/(1-\alpha\overline u)}.$$
All operators in  ${\mathcal B}^\alpha$ are of the latter form, so to check if $A^u_{\phi/(1-\alpha\overline u)}$ is Schatten class or compact, it suffices to examine $A_\phi^{u_\alpha}$ using the previous results.

\end{remark}

\section{Representing the Operators as Matrices}
Any TTO $A$ has a symbol $\phi+\overline{\psi}$ with $\phi$, $\psi$ holomorphic and thus $A=A_\phi+A_{\overline \psi}=A_\phi+A_\psi^*$. Applying the results of the previous section to both $A_\phi$ and $A_\psi$ gives sufficient conditions for $A$ to be in $S_p$. However, absent further structure allowing us to pass information effectively from $A$ to the summands, these conditions are not necessary. This is made clear by the results of Baranov et al.\ \cite{Baranov2010}. For both model spaces generated by Blaschke products and model spaces generated by singular inner functions they produce examples of bounded, rank one TTOs which have no bounded symbol. Recalling Sarason's result that a TTO with holomorphic symbol is bounded if and only if it has a bounded holomorphic symbol we see that this $A$ cannot be split as above with bounded summands. Noting that $A$ can be chosen to be rank one, we see further that even if $A$ is in all $S_p$ it need not be true that $A$ can be split as above with the summands in any $S_p$. 

We now consider the case where $u$ is an infinite Blaschke product. (Finite Blaschke products produce finite-dimensional model spaces, where all operators are trivially bounded, compact, and in all $S_p$ ideals.) We will assume all its zeros are simple and denote them by $z_n$ and the corresponding Blaschke factors by $b_n$. Let $k_n$ be the reproducing kernel at $z_n$ and $\hat k_n$ be the normalized reproducing kernel. Explicitly, for general $z\in \mathbb D$,
$$k_{u,z}(w) =  \frac{1-\overline{u(z)}u(z)}{1-\overline{z}w}.$$
In particular, if $z=z_n$,
$$k_n(w)=k_{u,z_n}(w)=\frac{1}{1-\overline{z_n}w}.$$
If $A=A_\phi+A_{\overline{\psi}}$ with $\phi+\overline{\psi}$ in $K_u+\overline{K_u}$, then we have a simple expression for the Berezin transform of $A$ at any $z_n$. It is 
$$\langle A_{\phi+\overline{\psi}} \hat k_n, \hat k_n\rangle = \langle (\psi + \overline{\psi}) \hat k_n, \hat k_n \rangle = \langle \phi +\overline{\psi}, |\hat k_n|^2\rangle = \langle \phi +\overline{\psi}, P_{z_n} \rangle_{H^2} = \phi(z_n) +\overline{\psi(z_n)},$$ 
where $P_{z_n}$ is the Poisson kernel for evaluating harmonic functions at $z_n$. Note this value is determined by the operator and is independent of the choice of symbol. Also, any function in $K_u$ is determined by its values on $\{z_n\}$. Thus, if $A$ has a holomorphic symbol, then, although the symbol is not completely determined by $\{\langle A\hat k_n, \hat k_n \rangle  \}$, the operator $A$ is determined. 

The functions $e_n=b_1\dots b_{n-1} \hat k_{n}$ form an orthonormal basis (see the remark following Theorem 10 in \cite{Garcia2013}). Note that when $\phi$ is analytic $A_\phi$ is upper triangular with respect to this basis. If $n_1 > n_2$, then 
$$\langle A_\phi e_{n_1} , e_{n_2} \rangle = \langle \phi  b_{n_2}\dots b_{n_1-1}\hat k_{n_1}, \hat k_{n_2}\rangle =0,$$
because we are evaluating at $z_n$ a product that includes $b_n$. The diagonal elements are exactly the $\phi(z_n)$: 
$$\langle A_\phi e_{n} , e_{n} \rangle = \langle \phi \hat k_n , \hat k_n \rangle  = \phi(z_n).$$
If $A_{\phi+\overline{\psi}}$ is in $S_1$, it has finite trace, so the norms of the diagonal elements must be absolutely summable. For the operator to be in $S_2$, the norms of the diagonal elements must be square summable, since to be in $S_2$ an operator must have square summable entries in any infinite matrix representation. We will show that under certain hypotheses, these necessary conditions are also sufficient. 

We let $u_n$ be $u$ with the $n$th factor omitted and define 
$$\delta_n = |u_n(z_n)| =  \prod_{i\ge 1, i\neq n} \left| b_i(z_n)\right|.$$

\begin{lemma}
Suppose $u$ is a Blaschke product. Let $\phi$ be an analytic symbol function such that 
$$\sum_{n=1}^\infty \frac{|\phi(z_n)|}{\delta_n} <\infty.$$
Then $A_\phi \in S_1$.
\end{lemma}
\begin{proof}
Note that $\alpha_n=u_n/u_n(z_n)$ takes the value $1$ at $z_n$ and $0$ at the other nodes. We estimate
$$\left \|\sum_{i=1}^\infty \phi(z_i)A_{\alpha_i}\right \|_{S_1}\le \sum_{i=1}^\infty |\phi(z_i)| \|A_{\alpha_i}\|_{S_1}=\sum_{i=1}^\infty \frac{|\phi(z_i)|}{\delta_i}\|A_{u_i}\|_{S_1}.$$
The modulus of $u_i$ is one on the circle and hence $\|A_{u_i}\|_{S_1}\le 1$, which shows that $B=\sum_{i=1}^\infty \phi(z_i)A_{\alpha_i}$ converges in $S_1$. It follows from Theorem 4.1 in \cite{Sarason2007} that the set of TTOs with holomorphic symbol are closed in $S_1$, so $B$ is also a TTO with holomorphic symbol. Because $B$ and $A_\phi$ have the same diagonal elements, it must be the case that $A_\phi=B$ and $A_\phi\in S_1$. 
\end{proof}
\begin{theorem}\leavevmode
\begin{enumerate}[(a)]
\item Suppose $u$ is an interpolating Blaschke product and $A_\phi$ is a TTO with analytic symbol. Then, for $1\le p \le \infty$, $A_\phi\in S_p$ if and only if $\{\phi(z_i)\}\in l^p$.
\item Suppose $u$ is an interpolating Blaschke product. For analytic $\phi$, $A_\phi$ is compact if and only if $\{\phi(z_n)\}$ tends to zero.
\end{enumerate}
\end{theorem}
\begin{proof}
(a) We prove the result for $p=1$ and $p=\infty$ and then finish using interpolation. It is a general fact that the map $D$ taking a $S_p$ operator $T$ to the sequence of diagonal elements $\{\langle Te_i, e_i\rangle \}$ is bounded from $S_p$ to $l_p$ for $p\in [1,\infty]$. It is well known that all $S_1$ operators have finite trace, and the standard proof of this fact shows that $D$ is bounded from $S_1$ to $l^1$ with norm $1$. The Cauchy-Schwarz inequality applied to the values $\langle Te_i, e_i\rangle$ shows that $D$ is bounded from $S_\infty$ to $l^\infty$ with norm 1. Then, by part (3) of Theorem 2.2.4, Theorem 2.2.6, and Theorem 2.2.7 of \cite{Zhu}, $D$ is bounded from $S_p$ to $l^p$ for every $p\in [1,\infty]$. 

We now consider the other direction. First recall that for $u$ an interpolating Blaschke product, the numbers $\delta_i$ are bounded away from zero. Hence the proof of the previous lemma exhibits a bounded map from $l^1$ to $S_1$. For the other endpoint, because $u$ is an interpolating sequence, given a bounded set of target values $\{z_n\}$ we can find a bounded holomorphic function $\phi$ that takes the targets values at the nodes. Thus $A_\phi$ is a bounded operator with the required diagonal matrix elements. Finally note that if $\tilde \phi$ is a different holomorphic function taking the same values on $\{z_n\}$ then $A_\phi=A_{\tilde \phi}$. Thus the map of $l^\infty$ into bounded TTOs is well defined and independent of the choice of symbol. Hence that map must be (the extension of) the map previously defined from $l^1$ to $S_1$. Thus the second part of the theorem also follows by interpolation.

(b) Since the $\{z_n\}$ form an interpolating sequence, the reproducing kernels $\hat k_n$ are a Riesz basis for $K_u$. Define $h_j=\|k_j \|u_j /\delta_j$. We compute 
$$\langle \hat k_j, u_j \rangle = \delta_j/\|k_j \|,$$
so the set $\{ h_j\}$ forms a dual basis to $\{\hat k_j\}$ and hence is also a Riesz basis. Further, the $h_j$ are eigenvectors for $A_\phi$ with eigenvalues $\phi(z_i)$:
$$ \langle A_\phi h_n,  k_j \rangle = \langle \phi h_n , k_j \rangle, $$
and this inner product is $0$ unless $n=j$, in which case the inner product is $\phi(z_i)\delta_j$. This implies that $A_\phi h_n = \phi(z_i)h_n$. Then $A_\phi$ is diagonalized by the Riesz basis $h_n$, and it is known that an operator diagonalized by a Riesz basis is compact if and only if the sequence of eigenvalues tends to zero. 
\end{proof}

If the zeros of a Blaschke product form an interpolating sequence, then the normalized reproducing kernels associated with those points are a Riesz basis for the model space. In such cases one can try to quantify how close those vectors are to being an actual orthonormal basis. One way to do this is to consider whether the vectors form an $S_p$ basis, a Schatten class perturbation of an orthonormal basis. Following Gorkin et al., we say that the normalized reproducing kernels at the nodes $\{\hat k_n\}$ form a $U+S_p$ basis if there exist $U$ unitary and $K\in S_p$ such that $\hat k_n = (U+K)e_n$ for all $n$, where $\{e_n\}$ is any orthonormal basis \cite{Gorkin}. Requiring that the reproducing kernels corresponding to a set of nodes $\{z_n\}$ form a $U+S_p$ basis is a stronger hypothesis than assuming the $\{z_n\}$ form an interpolating sequence. Gorkin et al.\ give a quantitative characterization of $S_p$ bases for $p\ge 2$. Their paper discusses Hardy space kernels, but since the formula for $k_n$ reduces to the Hardy space kernel on the nodes $z_n$, their results apply as stated to model spaces generated by Blaschke products. 
\begin{theorem}
Let be $u$  a Blaschke product and let $A$ be a TTO that admits a general symbol $f\in H^\infty +\overline{H^\infty}$. (Equivalently, $A$ splits as the sum of two bounded operators $A_\phi+A_{\overline{\psi}}$. Since any bounded TTO with holomorphic symbol admits a $H^\infty$ symbol, we may assume $\phi, \psi \in H^\infty$.) 
\begin{enumerate}[(a)]
\item Fix some $p\in [1,\infty)$ and suppose that the functions $\{\hat k_n\}$ form a $U+S_p$ basis. Then $A \in S_p$ if and only if $\{f(z_n)\}\in l^p$. 
\item Suppose that the functions $\{\hat k_n\}$ form a $U+S_\infty$ basis. Then $A$ is compact if and only if $\{f(z_n)\}$ tends to zero.
\end{enumerate}
\end{theorem}
\begin{proof}
We will move freely between operators and their representations as infinite-dimensional matrices. Fix some basis $e_n$. By hypothesis, there exists unitary $U$ and $K\in S_p$ such that $\hat k_n = (U+K)e_n$. With respect to the basis $Ue_n$, $A$ has matrix representation $[\langle A Ue_i, Ue_j \rangle]_{ij}$. We write $Ue_i=\hat k_i - Ke_i$. Then 
$$[\langle A Ue_i, Ue_j \rangle]=[\langle A(\hat k_i - Ke_i), (\hat k_j- Ke_j)]$$ $$=[\langle A\hat k_i, \hat k_j\rangle] - [\langle A\hat k_i, Ke_j] - [\langle Ke_i, \hat k_j \rangle ] + [\langle AKe_i, Ke_j\rangle ].$$
We will show that the first term is in $S_p$ if and only if $\{f(z_n)\}\in l^p$ and that the other three terms are always in $S_p$. This will complete the proof. 
We see that
$$[\langle A\hat k_i, \hat k_j\rangle]=[A_\phi \hat k_i +A_{\overline{\psi}}\hat k_i, \hat k_j]=[\langle A_{\overline{\psi}}\hat k_i, \hat k_j\rangle ] + [\langle \hat k_i, A_{\overline{\phi}} \hat k_j\rangle ]=\overline{\psi(z_i)}[\langle\hat  k_i, \hat k_j\rangle ] +\phi(z_j)[\langle \hat k_i, \hat k_j \rangle].$$
Let $D_\phi$ be the diagonal matrix with entries $\phi(z_i)$ and define $D_{\overline{\psi}}$ analogously. Note that these diagonal matrices are bounded by hypothesis. Let $G$ be the Gram matrix. Then
$$[\langle A\hat k_i, \hat k_j\rangle]=D_{\overline{\psi}}G+GD_\phi.$$
The hypothesis that the reproducing kernels form a $U+S_p$ basis implies that there exists $J\in S_p$ such that $G=I+J$ \cite{Gorkin}.  Then 
$$[\langle A\hat k_i, \hat k_j\rangle]=D_{\overline{\psi}}(I+J)+(I+J)D_\phi.$$
This is equal to an $S_p$ operator plus $D_{\overline{\psi}}+D_\phi$, which proves our claim about the first term. 

Consider now the second term. Under our hypotheses, $V=U+K$ is invertible, and we have
$$[\langle A\hat k_i, Ke_j] = [\langle \hat k_i ,A^*Ke_j\rangle]= [\langle V^{-1}\hat k_i, V^* A^*Ke_j \rangle] = [e_i, V^* A^*Ke_j \rangle].$$
This is the matrix of the $S_p$ operator $V^* A^*K$. Similarly, the third term is always in $S_p$. The fourth term is the matrix of the $S_p$ operator $K^*AK$. 

The argument for part (b) is similar. 
\end{proof} 

\begin{remark}
Using this theorem, we see that even if $Z=\left\{
z_{n}\right\}  $ satisfies the strong separation condition that the $\left\{
\hat{k}_{n}\right\}  $ be a $U+S_{p}$ basis, it may not be possible to split a
TTO in $S_{p}$ into a sum of two TTO's in that class, one with a holomorphic
symbol, the other with conjugate holomorphic symbol. Examples are obtained by
selecting the symbol $f=\phi+\bar{\psi}$ with $\phi,\psi$ holomorphic,
$\left\{  \phi(z_{n})+\bar{\psi}(z_{n})\right\}  \in\ell^{p}(Z)$ and
$\{\phi(z_{n})\},$ $\{\bar{\psi}(z_{n})\}\notin\ell^{p}(Z).$
\end{remark}

We now give a result for model spaces generated by singular inner functions. Our main tool will be results about the action of the operation of triangular projection on the Schatten classes originally due to Gohberg, Krein, Brodskii, and Macaev. We will use the presentation and formalism of \cite{Erdos1978}. Let $H$ be any Hilbert space. Given a finite nest $N$ of subspaces $0=S_0\subset \dots \subset S_n=H$, we let $P_{S_i}$ be the orthogonal projection onto $S_i$ and define $\triangle P_{S_i}=P_{S_i}-P_{S_{i-1}}$. Then, given some bounded operator $A$, we define 

$$\mathcal T_N(A)=\sum P_{S_{i-1}}A\triangle P_{S_i}$$
$$\mathcal R_N (A)= \sum P_{S_i}A\triangle P_{S_i}$$
$$\mathcal D_N(A)=\sum \triangle P_{S_i}A\triangle P_{S_i}.$$

If we think of $A$ as having a block matrix representation where the blocks correspond to the subspaces in the partition, then $\mathcal T_N(A)$ is the strictly upper triangular part of the matrix, $\mathcal R_N (A)$ is the upper triangular part including the diagonal blocks, and $\mathcal D_N(A)$ is the diagonal. Thus $\mathcal R_N = \mathcal T_N +\mathcal D_N$. Theorem 3.2 in \cite{Erdos1978} shows that if $p\in(1,\infty)$ and $A\in S_p$, then each of the above nets converges in $S_p$ as the nest is refined.

In the proof of next theorem, we construct a continuous nest of subspaces with respect to which a TTO with holomorphic symbol becomes upper triangular and a TTO with anti-holomorphic symbol becomes lower triangular. By the above discussion, the projection onto the upper triangular part is bounded in $S_p$. Then it seems that applying this projection to an arbitrary TTO $A_{\phi + \overline{\psi}}$ gives $A_\phi$, and hence that we can split $S_p$ TTOs into their holomorphic and anti-holomorphic parts boundedly. Unfortunately, this is not the case. In general, both the holomorphic and the anti-holomorphic parts contribute to the limit of the diagonal net, so the upper triangular projection recovers the holomorphic part plus some error term. However, by imposing a suitable hypothesis on the symbol, we can insure this error term is zero, so that this scheme works; that is, we can write $A$ as a sum of two operators and use the results of section 2 on each to obtain necessary and sufficient conditions for $A$ to be in $S_p$. 

If $u$ is a singular inner function, we will use $P_r$ to mean $P_{K_{u^r}}$.

\begin{theorem}
Let $u$ be an inner function and $p\in (1, \infty)$. Let $A$ be a $S_p$ TTO with canonical symbol $f=\phi+\overline \psi \in K_u + \overline K_u$ such that, for some $r \in (0,1)$, $A_{P_r f} \in S_p$ (equivalently, $A_{P_r \phi } \in S_p$). Then $A_\phi$ and $A_{\overline\psi}$ are both $S_p$ operators.
\end{theorem}
\begin{proof}
We may assume without loss of generality that $A_{P_r f}=0$ by noting that $A\in S_p$ if and only if $A-A_{P_r f}$ is. 
 
Define $K_a = K_{u^a}$. For $a\in [0,1]$, this gives a continuous nest of subspaces 
$$K_0 = 0  \subset \dots \subset K_a \subset \dots \subset K_1 = K_u.$$
We have $A_{\overline{\psi}} K_a \subset K_a$, so for any finite subnest, the corresponding block matrix for $A_{\overline{\psi}}$ is upper triangular. By taking adjoints, we see that the block matrix for $A_\phi$ is lower triangular. We have 
$${\mathcal R}_N(A)={\mathcal R}_N(A_\phi+A_{\overline{\psi}})={\mathcal D}_N(A_\phi) + A_{\overline{\psi}}.$$
It suffices to show that ${\mathcal D}_N(A_\phi)$ tends to zero. By taking adjoints again, it is enough to show that ${\mathcal D}_N(A_{\overline{\phi}})$ tends to zero. Choose a finite nest  ${K_{q_i}}$, with $q_0=0$, $q_n=1$ and $0\le i \le n$. We further stipulate that $|q_i-q_{i-1}|<r$ for all $i>0$. We will show that $\triangle K_{q_i} A_{\overline{\phi}} \triangle K_{q_i}$ vanishes for every $i$, so that ${\mathcal D}_N(A_{\overline{\phi}})$ is zero for all sufficiently fine partitions. 

Under our hypotheses, we can write $\phi = u^r \phi_1$ for some holomorphic $\phi_1$. This shows that $A_{\overline{\phi}} K_{q_i} \subset K_{q_{i-1}}$. Hence $\triangle K_{q_i} A_{\overline{\phi}} \triangle K_{q_i}$ vanishes for every $i$.
\end{proof}

A similar theorem holds when $A_{P_r \overline{f}}\in S_p$. In section 6 of \cite{Baranov2010} the authors consider, among other things, conditions that allow effective splitting of a TTO into an analytic and conjugate analytic component. The technical hypotheses they impose are similar in spirit to those in the previous theorem.

In the next several results we show that a similar splitting is possible for some products of singular inner functions. In the case of an atomic singular inner function, such a splitting into $S_p$ summands is always possible. Rochberg proved this in \cite{Rochberg1987}.

For the remainder of this section, we let $u$ and $v$ be two atomic singular inner functions with single disjoint atoms. The functions $u$ and $v$ form a corona pair, so we may apply the corona theorem. Then there exist $a,b\in H^\infty$ such that $1=au+bv$. We let $B_\phi$ and $C_\phi$ denote TTOs on the spaces $K_v$ and $K_u$ respectively. 
\begin{lemma}
Let $A$ be a bounded TTO on $K_{uv}$ that admits a symbol $f\in H^\infty$. Then $A\in S_p(K_{uv})$ if and only if $B_{auf}\in S_p(K_v)$ and $C_{bvf}\in S_p(K_u)$. 
\end{lemma}
\begin{proof}
Suppose that $A\in S_p$. Note that $K_{uv}=K_u\oplus u K_v$, so the compression of $A$ to $K_u$ is in $S_p$. We have $A_f=A_{auf+bvf}$, and because $K_u$ is orthogonal to all multiples of $u$, the compression to $K_u$ is exactly $C_{bvf}$. Similar reasoning shows that $B_{auf}\in S_p$.

Now suppose that $B_{auf}$ and $C_{bvf}$ are both $S_p$ operators on their respective spaces. Let $g=aug+bvg$ be an arbitrary element of $K_{uv}$. We see 
$$A_f(g)=P_{K_{uv}}[(aug+bvg)(auf+bvf)]=P_{K_{uv}}[a^2u^2fg+b^2v^2fg]=uP_v[a^2ufg]+vP_u[b^2vfg]$$
$$=uB_{auf}B_a(g)+ vC_{bvf}C_b(g).$$
This is a sum of $S_p$ operators, so it is in $S_p$.
\end{proof}
\begin{theorem}
Let $A$ be a TTO on $K_{uv}$ with bounded symbol $f=\phi+\overline \psi$. Then $A\in S_p(K_{uv})$ if and only if $A_\phi \in S_p(K_{uv})$ and $A_{\overline \psi}\in S_p(K_{uv})$. 
\end{theorem}
\begin{proof}
Write $f=ua\phi+ vb\phi + \overline {ua\psi} + \overline{vb\psi}$. The the compression of $A$ to $K_u$ is also in $S_p$. This compression is $B_{\overline{u a\psi}} +B_{v b\phi} + B_{\overline{v b \psi}}$. Note that $B_{\overline{u a \psi}} =0$ on $K_u$. Then $B_{v b\phi} + B_{\overline{vb\psi}}$ is an $S_p$ operator, and using the results of \cite{Rochberg1987} on the space $K_u$, we see that $B_{vb\phi}\in S_p$ on $K_u$. Similar reasoning shows that $C_{ua\phi}\in S_p$ on $K_v$. Then the previous lemma shows that $A_\phi\in S_p$. 
\end{proof}

We note that the previous result extends to any corona pair of inner functions $u$ and $v$ and by induction to any $n$ inner functions that are valid corona data.

\section{Polynomial Symbol}
In \cite{Ahern1970}, Ahern and Clark prove that any model space is unitarily equivalent to a sum of three $L^2$ spaces of a simple form. Specifically, they prove the following result. Suppose that $u$ is an inner function with canonical decomposition $u=Bs\triangle$, where
$$B(z)=\prod_{n=1}^\infty \left(-\frac{\overline{a_n}}{|a_n|} \right) \frac{z-a_n}{1-\overline{a_n}z},$$
$$s(z) = \exp \left( -\int_0^{2\pi} \frac{e^{i\theta} + z }{e^{i\theta} - z}\ d\sigma(\theta)\right),$$
$$\triangle(z) = \exp\left( - \sum_{n=1}^\infty r_n \frac{e^{i\theta_n}+z}{e^{i\theta_n} - z } \right).$$
Here $a_n$ is a Blaschke sequence (we define $\overline{a_n}/|a_n|=1$ when $a_n=0$), $\sigma$ is a finite, nonnegative, continuous, singular measure, and the $r_n$ are nonnegative with $\sum r_n < \infty$. 

Then there is a unitary operator $V$ taking $K_u$ to $L^2(d\sigma_B) \oplus L^2 (d\sigma) \oplus L^2(d\tau)$, where $\theta_B$ is the measure on the positive integers with mass $1-|a_k|$ at $k$ and $\tau$ is the measure on $[0,\infty)$ defined by $\tau = r_{n+1} m$ on the real interval $[n, n+1)$ and $m$ is the Lebesgue measure. Further, $V$ takes the TTO $A_z$ to an operator of the form $M+H$, where $M$ is a multiplication operator and $H$ is a Hilbert-Schmidt operator. The operator $M$ has the form
$$M=M_B\oplus M_s \oplus M_\triangle,$$
where $(M_B f)(n)=z_n f(n)$, $(M_sf)(\lambda)= e^{i\lambda}f(\lambda)$ and $(M_\triangle f)(\lambda)=e^{ir_{n+1}}f(\lambda)$ for ${\lambda\in [n,n+1)}$.

\begin{theorem}
Suppose the symbol function $q\in L^2(\partial \mathbb D)$ is a polynomial in $z$ and $\overline z$.
\begin{enumerate}[(a)]
 \item If $u$ is a Blaschke product with zeros $z_k$, then $A_q \in S_p$ for $p\ge 2$ if and only if $\{q(z_i)\}\in l^p$, and $A$ is compact if and only if $\{q(z_i)\}$ tends to zero.
\item If $u$ is a continuous singular function, then $A_q$ is compact if and only if $q$ vanishes on the support of $\sigma$, and in this case it is the zero operator. 
\item If $u$ is an atomic singular function, then $A_q$ is compact if and only if $q(e^{iu_{n}})=0$ for every $u_n$, and in this case $A_q$ is the zero operator. 
\end{enumerate}
\end{theorem}
\begin{proof}
\begin{enumerate}[(a)]
\item It is clear that $A_z A_z = A_{z^2}$, and $A_{z^2}=M^2 + H'$, where $H'$ is Hilbert-Schmidt. From this we deduce that if $p$ is a polynomial in $z$ and $\overline z$, then modulo some Hilbert-Schmidt operator, $A_q$ is a multiplication operator $M$ given by $Mf=q(z_n) f(n)$. Then $A\in S_p$ for $p\ge 2$ if and only if $M$ is, and $A$ is compact if and only if $M$ is. Hence it suffices to study $M$, which is diagonal with entries $\{q(z_i)\}$. Then $M$ is in $S_p$ for $p\ge 2$ if and only if this sequence is in $l^p$, and $M$ is compact if and only if $\{q(z_i)\}$ tends to zero. 
\item Here $M$ takes the form $(Mf)(\lambda)= q(e^{i\lambda})f(\lambda)$, and the only compact multiplication operator is the zero operator. 
\item Here $M$ is $(Mf)(\lambda)=q(e^{ir_{n+1}})f(\lambda)$, and the only compact multiplication operator is the zero operator. 
\end{enumerate}
\end{proof}

One naturally wonders if the above results extend to more general functions by taking limits. We note that using this idea, Ahern and Clark showed that when $f$ is continuous on $\partial D$, $A_f$ is compact if and only if $f(e^{i\theta})=0$ for all $e^{i\theta}\in \operatorname{supp} u \cap \partial \mathbb D$ (Theorem 5.4 in \cite{Ahern1970}).
\section{Hilbert-Schmidt Truncated Hankel Operators}

We noted in Lemma 1 that the Toeplitz and Hankel theories are closely intertwined. The Berezin transform is known to be a powerful tool in studying Toepliz operators. The Berezin transform cannot be defined for Hankel operators because they map one space into a different space. However, the fact that they map the space into its dual space leads to a natural, closely related construction; namely taking $B_f$ to the function $\langle B_f k_\zeta, \overline{k_\zeta}\rangle$, which, in contrast to the Berezin transform, is an analytic function of $\zeta$. Note that the Berezin transform is traditionally defined using normalized kernels, and we could have used normalized kernels here and gotten similar, if slightly more complicated, formulas than the ones below.

	We  present our characterization of Hilbert-Schmidt TTOs in two parts. The proof of the first goes through for any sufficiently nice reproducing kernel Hilbert space, which shows a more general fact about such spaces: the Hilbert-Schmidt norm of a bilinear form of Hankel type is $\langle f , Tf\rangle$, where $Tf=\langle f, k^2_\zeta\rangle$.

\begin{theorem}
Suppose $B_{\bar{f}}$ is a truncated Hankel operator with conjugate analytic
symbol, and set ($Tf)(\zeta)=\left\langle f,k_{\zeta}^{2}\right\rangle .$ Then
$B_{\bar{f}}\in S_{2}$ if and only if $\left\langle f,Tf\right\rangle $ is
finite. In fact%
\[
\left\Vert B_{\bar{f}}\right\Vert _{S_{2}}^{2}=\left\langle f,Tf\right\rangle
.
\]

\end{theorem}

\begin{proof}
The map $\left(  \alpha,\beta\right)  \rightarrow\left\langle B_{\bar{f}%
}\alpha,\bar{\beta}\right\rangle =\left\langle \alpha\beta,f\right\rangle $
defines a bilinear functional on $K_{u}$ and hence a linear functional $L$ on
the algebraic tensor product $K_{u}\otimes_{\text{alg}}K_{u}.$ Having
$B_{\bar{f}}\in S_{2}$ is equivalent to knowing that the functional extends
continuously to the Hilbert space tensor product $K_{u}\otimes K_{u}.$ As such
it is of the form $L(\alpha\otimes\beta)=\left\langle \alpha\otimes
\beta,\mathfrak{B}_{f}\right\rangle $ for some $\mathfrak{B}_{f}\in
K_{u}\otimes K_{u}.$ The functional $L$ is ``of Hankel type," that is, for any
$\alpha,\beta,$ $L(\alpha\otimes\beta)$ is a linear functional of the
pointwise product $\alpha\beta.$ Hence, by Proposition 2.1 of \cite{Peetre1995}, $L\perp
V_{D},$ where $V_{D}$ is the subspace of $K_{u}\otimes K_{u}$ of functions
which vanish on the diagonal; $N\in V_{D}$ exactly if $N(k_{\zeta},k_{\zeta
})=0$ for all $\zeta$ in the disk. Hence the norm of $L$ equals the norm of
its image in the quotient space $V_{D}^{\perp}.$ 

The space $V_{D}^{\perp}$ is described by a classical result of Aronszajn
(Theorem II in Section 8 of \cite{Aronszajn1950}). He showed that $V_{D}^{\perp}$ is (isometrically isomorphic
as a reproducing kernel Hilbert space to) $H\left(  k_{\zeta}^{2}\right),$ 
 the reproducing kernel Hilbert space of functions on the unit disk generated
by the kernel functions $\left\{  k_{\zeta}^{2}\right\}  .$ Using this
identification, we see that the norm of $\mathfrak{B}_{f}$ in the tensor
product, and hence the norm of the operator in $S_{2},$ is the norm of the
restriction of $\mathfrak{B}_{f}$ to the diagonal; that is, the norm in
${H\left(  k_{\zeta}^{2}\right)} $ of $\left\langle \mathfrak{B}_{f},k_{\zeta
}\otimes k_{\zeta}\right\rangle =(Tf)(\zeta)$. Summarizing,
\[
\left\Vert B_{\bar{f}}\right\Vert _{S_{2}}^{2}=\left\langle Tf,Tf\right\rangle
_{H(k_{\zeta}^{2})}.
\]

This gives the norm of $B_{\bar{f}},$ but we would prefer an answer that does
not involve the inner product in $H\left(k_{\zeta}^{2}\right)$, about which we know very
little. To accomplish that and complete the proof we show that for all
$\alpha,\gamma\in K_{u}$ we have $\left\langle \alpha,\mathcal{\gamma
}\right\rangle _{K_{u}}=\left\langle T\alpha,\mathcal{\gamma}\right\rangle
_{H\left(k_{\zeta}^{2}\right)}.$ By linearity it suffices to consider the case of
$\alpha=k_{w}$ a reproducing kernel, because linear combinations of such kernels are dense. We first compute
\[
Tk_{w}(\zeta)=\left\langle k_{w},k_{\zeta}^{2}\right\rangle =\overline
{k_{\zeta}^{2}(w)}=k_{w}^{2}(\zeta).
\]
Hence
\[
\left\langle Tk_{w},\mathcal{\gamma}\right\rangle _{H(k_{\zeta}^{2}%
)}=\left\langle k_{w}^{2}(\zeta),\mathcal{\gamma}\right\rangle _{H(k_{\zeta
}^{2})}=\overline{\left\langle \mathcal{\gamma},k_{w}^{2}(\zeta)\right\rangle
_{H(k_{\zeta}^{2})}}=\overline{\gamma(w)}.%
\]
On the other hand
\[
\left\langle k_{w},\mathcal{\gamma}\right\rangle _{K_{u}}=\overline
{\left\langle \mathcal{\gamma},k_{w}\right\rangle _{K_{u}}}=\overline
{\gamma(w)}.%
\]

\end{proof}

We have reduced the problem to determining when $\langle f, Tf \rangle$ is finite. The next theorem computes $Tf$ for model spaces, completing our characterization of those THOs. We write $f=f_1 + uf_2$ with $f_i\in K_u$.

\begin{theorem}
We have
$$Tf(w)=\langle f, k^2_{u,w} \rangle = (zf)'(w) -2u(w)(zf_2)'(w).$$
\end{theorem}
\begin{proof}
Recall $Tf(w)=\langle f, k^2_{u,w} \rangle$. We have 
$$k^2_{u,w}(z)=\left( \frac{1-\overline u(w) u(z)}{1-\overline w z}\right)^2 = \frac{1}{(1-\overline w z)^2}  -2 \frac{\overline u(w)u(z)}{(1-\overline w z)^2} +\frac{\overline u(w)^2u(z)^2}{(1-\overline w z)^2}.$$
The third term is orthogonal to $K_{u^2}$, so we may ignore it. The inner product of $f$ against the first term is 
$$\left\langle f, \frac{1}{(1-\overline w z)^2 } \right\rangle = (wf)'.$$
The inner product against the second term is 
$$-2 \left\langle f,  \frac{\overline u(w)u(z)}{(1-\overline w z)^2} \right\rangle  = -2 \left\langle  \overline u(z)(f_1+u(z)f_2),  \frac{\overline u(w)}{(1-\overline w z)^2} \right\rangle$$ $$=-2u(w)\left\langle b_2, \frac{1}{(1-\overline w z)^2}\right\rangle = -2u(w)(wf_2)'.$$
The total is 
$$Tf(w)=\langle f, k^2_{u,w} \rangle = (zf)'(w) -2u(w)(zf_2)'(w).$$
\end{proof}
\begin{remark}
When $f=f_1$ or $u=0$, the condition that $\langle f, Tf \rangle $ is finite is exactly the usual condition for a Hankel operator on the Hardy space to lie in $S_2$: $f$ must lie in the Dirichlet space. In fact, this is a restatement of the observation behind the proof of Theorem 1; if a TTO has an analytic symbol, then it can be understood using the classical theory of Hankel operators on the Hardy space. Note also that computing $T$ for the Hardy and Fock spaces and applying the previous theorem recovers the classical results for these spaces.
\end{remark}

These results suffice to prove a complete characterization of the Hilbert-Schmidt TTOs, since by Lemmas $1$ and $2$, any TTO $A_{\overline{\psi} + \phi}$ may be written as the truncated Hankel operator $B_{\overline{C\phi +\nu u}}$ and Theorem $8$ may be applied. 

\bibliography{SchattenTTOs}{}

\begin{thebibliography}{10}

\bibitem{Ahern1970}
P.R. Ahern and D.N. Clark.
\newblock {On Functions Orthogonal to Invariant Subspaces}.
\newblock {\em Acta Mathematica}, 124(1):191--204, 1970.

\bibitem{Aronszajn1950}
N.~Aronszajn.
\newblock {Theory of Reproducing Kernels}.
\newblock {\em Transactions of the American Mathematical Society},
  68(3):337--404, 1950.

\bibitem{Baranov2010}
A.~Baranov, I.~Chalendar, E.~Fricain, J.~Mashreghi, and D.~Timotin.
\newblock {Bounded Symbols and Reproducing Kernel Thesis for Truncated Toeplitz
  Operators}.
\newblock {\em Journal of Functional Analysis}, 259(10):2673--2701, 2010.

\bibitem{Bessonov2014}
R.~V. Bessonov.
\newblock {Fredholmness and Compactness of Truncated Toeplitz and Hankel
  Operators}.
\newblock {arXiv:1407.3466 [math.CV]}.

\bibitem{Erdos1978}
J.A. Erdos.
\newblock {Triangular Integration on Symmetrically Normed Ideals}.
\newblock {\em Indiana University Mathematics Journal}, 27(3):401--408, 1978.

\bibitem{Garcia2013}
S.R. Garcia and W.T. Ross.
\newblock {Recent Progress on Truncated Toeplitz Operators}.
\newblock In {\em Blaschke Products and Their Applications}, pages 275--319.
  2013.

\bibitem{Gorkin}
P.~Gorkin, J.E. McCarthy, S.~Pott, and B.D. Wick.
\newblock {Thin Sequences and the Gram Matrix}.
\newblock {\em Arch. Math.}, 103:93--99, 2014.

\bibitem{Janson1987}
S.~Janson, J.~Peetre, and R.~Rochberg.
\newblock {Hankel Forms and the Fock Space}.
\newblock {\em Rev. Mat. Iberoamericana}, 3(1):61--138, 1987.

\bibitem{Nikolski2002}
N.K. Nikolski.
\newblock {\em {Operators, Functions, and Systems: An Easy Reading}}.
\newblock American Mathematical Society, 2002.

\bibitem{Peetre1995}
J.~Peetre and R.~Rochberg.
\newblock {Higher Order Hankel Forms}.
\newblock In {\em Multivariable Operator Theory}, pages 283--306, 1995.

\bibitem{Rochberg1987}
R.~Rochberg.
\newblock {Toeplitz and Hankel Operators on the Paley-Wiener Space}.
\newblock {\em Integral Equations and Operator Theory}, 10(2):187--235, 1987.

\bibitem{Sarason2007}
D.~Sarason.
\newblock {Algebraic Properties of Truncated Toeplitz Operators}.
\newblock {\em Oper. Matrices}, 1(4):491--526, 2007.

\bibitem{Sedlock2010}
N.A. Sedlock.
\newblock {Algebras of Truncated Toeplitz Operators}.
\newblock {\em Oper. Matrices}, 5(2):309--326, 2011.

\bibitem{Semmes1984}
S.~Semmes.
\newblock {Trace Ideal Criteria for Hankel Operators, and Applications to Besov
  Spaces}.
\newblock {\em Integral Equations and Operator Theory}, 7(2):241--281, 1984.

\bibitem{Zhu}
K.~Zhu.
\newblock {\em {Operator Theory in Function Spaces}}.
\newblock American Mathematical Society, 2nd edition, 2007.

\end{thebibliography}
\bibliographystyle{plain}
\end{document}